\documentclass{birkjour}
\usepackage{amsmath,amssymb}
\usepackage{xcolor}
\usepackage{frcursive}
\usepackage{comment}
\usepackage{titlesec}
\usepackage[utf8]{inputenc}
\usepackage{amssymb}
\usepackage{enumitem}
\usepackage[hidelinks,bookmarks=true]{hyperref}
\usepackage[numbered]{bookmark}

\makeatletter
\newcommand*{\sectionbookmark}[1][]{%
  \bookmark[%
    level=section,%
    dest=\@currentHref,%
    #1%
  ]%
}
 \newtheorem{thm}{Theorem}[section]
 \newtheorem{cor}[thm]{Corollary}
 \newtheorem{lem}[thm]{Lemma}
 \newtheorem{prop}[thm]{Proposition}
 \theoremstyle{definition}
 \newtheorem{defn}[thm]{Definition}
 \theoremstyle{remark}
 \newtheorem{rem}[thm]{Remark}
 
 \newtheorem{proofpart}{Step}[thm]

 \numberwithin{equation}{section}

\newcommand{\N}{\mathbb{N}}

\newcommand{\R}{\mathbb{R}}

\newcommand{\Om}{\Omega}

\newcommand{\ric}{\mathrm{Ric}}

\newcommand{\n}{\mathbf{n}}

\newcommand{\dist}{\mathrm{dist}}

\newcommand{\vol}{\mathrm{Vol_g}}

\newcommand{\diam}{\mathrm{diam}}

\let\emptyset\varnothing

\begin{document}

%
%
%
%
%
%
%
%
%

\title[Eigenvalues of the Wentzel-Laplace operator]{
\begin{center}
 Asymptotically sharp bound for Wentzel-Laplace eigenvalues
\end{center}}

\author[A. Ndiaye]{
\begin{center}
A\"issatou M. NDIAYE
\end{center}}

\address{%
 Institut de math\'ematiques\\
			Université de Neuchâtel\\
			   Switzerland \\
		     Tel.: +41327182800\\
         }
           
 \email{aissatou.ndiaye@unine.ch}

%

\date{\today}

\begin{abstract}
We prove  asymptotically optimal upper bounds  for the eigenvalues of the Wentzel–Laplace operator on Riemannian manifolds with Ricci curvature bounded below. These bounds depend
 highly on the geometry of the boundary in addition to the dimension and the volume of the manifold.
\end{abstract}

\maketitle

\section{Introduction}
Let $n\geqslant 2$ and $(M,g)$ be a complete Riemannian manifold of dimension $n$. Let $\Omega\subset M$ be a bounded domain with smooth boundary $\Gamma$. 
We denote by $\Delta$ and  $\Delta_\Gamma$ the Laplace-Beltrami operators  acting on functions  on $\Omega$ and $\Gamma$ respectively. Given  an arbitrary
constant $\beta\in\R_{\geqslant 0}$, consider the following eigenvalue problem on $\Omega$:
\begin{equation}\label{w}
\begin{cases}
\Delta u=0 \quad \text{in}~\Omega,\\
\beta \Delta_\Gamma u +\partial_\n u=\lambda u  \quad \text{on}~\Gamma.\\
\end{cases}\quad\text{(Wentzel Problem)},
\end{equation}
where $\partial_\n$ denotes the outward unit normal derivative.

The spectrum of the Laplacian with Wentzel boundary condition consists in an increasing countable sequence of eigenvalues
 \begin{equation}\label{spectrum}
 0=\lambda_{W,0}^{\beta}< \lambda_{W,1}^{\beta}\leqslant\lambda_{W,2}^{\beta}\leqslant \cdots\leqslant \lambda_{W,k}^{\beta}\leqslant \cdots \nearrow \infty.
 \end{equation}
 We adopt the convention that each eigenvalue is repeated according to its multiplicity.
 Let $\mathfrak{V}(k) $ denote  the set of  all $ k$-dimensional  subspaces of   $ \mathfrak{V}_\beta$ which  is defined by
\begin{equation}
 \mathfrak{V}_\beta \stackrel{\scriptscriptstyle\text{def}}=\{(u,u_\Gamma)\in H^1(\Omega)\times H^1(\Gamma): u_\Gamma=u|_\Gamma \}.
\end{equation}
For every $k\in\N$, the $k$th eigenvalue of the Wentzel-Laplace operator $B_\beta$ satisfies

\begin{equation}\label{char}
 \lambda_{W,k}^{\beta}(\Omega)={\underset{V\in \mathfrak{V}(k)}{\min}  }\underset {0\neq u\in V} {\max} R_\beta(u),
\end{equation}
where $ R_\beta(u) $,  the Rayleigh quotient for $\mathrm{B}_\beta$, is given by 
\begin{equation}\label{rayleigh}
 R_\beta(u) \stackrel{\scriptscriptstyle\text{def}}=\frac{\int_\Om{|\nabla u|^2 \mathrm{d}_M+\beta\int_{\Gamma}{|\nabla_\Gamma u|^2 \mathrm{d}_\Gamma}}}{\int_{\Gamma}{u^2 \mathrm{d}_\Gamma}}, \quad \text{for all } u\in \mathfrak{V}_\beta\backslash\{0\}.
\end{equation}
We obtain a relevant upper bound for the  eigenvalues of the problem \eqref{w}, according to the Weyl law:
 \begin{equation}
\label{WeylW}
\lambda_{W,k}^{\beta}(M,g)=\beta C_n^2k^{\frac{2}{n-1}}+O(k^{\frac{2}{n-1}}),\quad k\rightarrow\infty,
\end{equation}
where $ C_n=\frac{2\pi}{\left( \omega_{n-1}\vol(\Gamma) \right)^{\frac{1}{n-1}}}$.
\begin{thm}\label{main29}
Let $n\geqslant 2$ and $(M, g)$ be a complete $n$-dimensional Riemannian manifold with $\ric (M,g) \geqslant -(n-1)\kappa^2$, with $\kappa\in \R_{\geqslant 0}$. Let $\Omega\subset M$ be a
domain with smooth boundary $\Gamma$. Then for
every $k\geqslant 1$, one has
%
%

\begin{align}\label{eq09042020a}
\lambda_{W,k}^{\beta}(\Om)
&\leqslant
A(n)\left[\kappa\frac{\vol(\Omega)}{\vol(\Gamma)}+\beta\right]\left( \frac{k}{\vol(\Gamma)}  \right)^{\frac{2}{n-1}}
\nonumber \\
&+{B}(n,\kappa,C_0)
  \left[ \left( \frac{\vol(\Om)}{\vol(\Gamma)}\right)^{1-\frac{2}{n}} +\frac{\vol(\Om)}{\vol(\Gamma)}+\beta \right] \left( \frac{k}{\vol(\Gamma)}  \right)^{\frac{2}{n}} 
\nonumber\\
&+{C}(n,\kappa,R_0) \left[ \frac{\vol(\Om)}{\vol(\Gamma)}+\beta \right],
\end{align}
where the constant $A(n)$ depends only on the dimension $n$,  $B(\Omega,\kappa, C_0)$ and $C(\Omega,\kappa,R_0)$ depend in addition on $\kappa$ and other geometric constants $C_0$ and $R_0$.
\end{thm} 
As an immediate corollary of this result, we have the following upper bound.
\begin{cor}\label{cor07042020}
Under the assumptions of Theorem \ref{main29}, we have
$$
\lambda_{W,k}^{\beta}(\Om)
\leqslant
  \left( A(n)\left[\kappa\frac{\vol(\Omega)}{\vol(\Gamma)}+\beta\right]+1\right)  \left( \frac{k}{\vol(\Gamma)}  \right)^{\frac{2}{n-1}}  
+{B}(\Omega,\beta),
$$
where the constant $A(n)$ depends only on the dimension $n$  and ${B}(\Omega,\beta)$ depends on geometric quantities of $\Omega$ and on $\beta$.
\end{cor}

\begin{rem}
It is important to remark that the hypotheses of Theorem \ref{main29}  are quite strong. They make part of unexpressed geometric properties  that are involved in Lemma \ref{Lnard} hidden behind the constants $R_0$ and $C_0$. 
Those required properties place  geometric constraints  on both $\Gamma$ and the ambient Riemannian manifold $M$ including a bounds on the second fundamental form and on the Ricci curvature of $\Gamma$.
However, we do not think that one can allow much weaker assumptions. Nonetheless, optimality of these assumptions  deserves to be discussed.
It would be interesting to compare them with alternative assumptions on the sectional curvature in $M$  and principal curvatures  of $\Gamma$.
\end{rem}
%

A general important assumption substantiated by the use of lemma \ref{Lnard} in the proof of Theorem \ref{main29} is roughly the required  control of "volume concentration" of $\Gamma$.
This might be comprehended trough the more theoretical result given in Theorem \ref{main21032020} below. 
A general sufficient assumption underlying this volume control is expressed in the following definition.
\begin{defn}\label{def28032020}
Let $n\geq 2$ and $\tilde{C}$ be a positive real number.
We designate by $\mathcal{M}(n,\tilde{C})$ the class of all n-dimensional  Riemannian manifolds  with boundary $\Gamma$ such that, for all $x \in\Gamma$ and all radius $0<r <1$, we have
$$\vol(B(x, r) \cap \Gamma) \leqslant \tilde{C}r^{n-1},$$
where $B(x, r)$ denotes the $n$-dimensional metric ball of center $x$ and radius $r > 0$.
\end{defn}
 This allows us to ceil the absorbed boundary volume by $n$-dimensional metric balls. The eigenvalues of manifolds in $\mathcal{M}(n,\tilde{C})$ are uniformly controlled. We establish the following very general result:
\begin{thm}\label{main21032020}
Let $n\geqslant 2$ and $\tilde{C}$ be a positive real number. Let $(M, g)$ be a complete $n$-dimensional Riemannian manifold with $\ric (M,g) \geqslant -(n-1)\kappa^2$, with $\kappa\in \R_{\geqslant 0}$. Let $\Omega\subset M$ be a domain with smooth boundary $\Gamma$ such that $\Omega\in \mathcal{M}(n,\tilde{C})$ . Then for
every $k\geqslant 1$, one has
\begin{align}\label{eq09042020b}
\lambda_{W,k}^{\beta}(\Om)
&\leqslant
A(n)\left[\kappa\frac{\vol(\Omega)}{\vol(\Gamma)}+\beta\right]\left( \frac{\tilde{C}k}{\vol(\Gamma)}  \right)^{\frac{2}{n-1}}
\nonumber \\
&+B(n,\kappa)
 \left( \frac{\vol(\Om)}{\vol(\Gamma)}  \right)^{1-\frac{2}{n}}
 \left( \frac{k}{\vol(\Gamma)}  \right)^{\frac{2}{n}} 
\nonumber\\
&+C(n,\kappa) \left(  \frac{\vol(\Om)}{\vol(\Gamma)}+\beta \right),
\end{align}
where  the constant $A(n)$ depends only on the dimension,   $B(n,\kappa)$ and $C(n,\kappa)$ depend only on the dimension $n$ and $\kappa$. 

\end{thm}
This general result yields the following corollary.
\begin{cor}\label{cor070420202}
Let the assumptions of Theorem \ref{main21032020} hold.  For
all $k\in \N$, one has
$$
\lambda_{W,k}^{\beta}(\Om)
\leqslant
\left( A(n)\left[\kappa\frac{\vol(\Omega)}{\vol(\Gamma)}+\beta\right]\tilde{C}^{\frac{2}{n-1}}+1\right)\left( \frac{k}{\vol(\Gamma)}  \right)^{\frac{2}{n-1}}+B(\Omega,\beta).
 $$
 where the constant $A(n)$ depends only on the dimension and  $B(\Omega,\beta)$ depends on the geometry ($n$, $\kappa$, $\vol(\Omega)$, $\vol(\Gamma)$) and on $\beta$.
\end{cor}
\begin{rem}
An invariant that measures the concentration of the volume as in Definition \ref{def28032020} is  aver
by the authors in \cite{dryden2010} with the intersection index.

If $\Omega$ is an Euclidean domain of $\R^n$ so that its boundary $\Gamma$ is a compact hypersurface, the intersection index of $\Gamma$is defined as the supremum number of transversal intersections of real lines with $\Gamma$:
$$i(\Gamma):=\sup \{\sharp(\Gamma\cap \pi), ~\pi \text{ transversal line to } \Gamma \}.$$
\end{rem}
%
%

From \cite[Prop 2.1]{dryden2010}, for every $x\in \R^n$ and $r>0$, one has  
$$\vol\left(\Gamma\cap B(x,r)\right)\leqslant \frac{i(\Gamma)}{2}\vol\left(\mathbb{S}^{n-1}\right)r^{n-1},$$
where $B(x, r)$ denotes the Euclidean ball of center $x$ and radius $r$ in $\R^n$ and $\mathbb{S}^{n-1}$ the standard
$n$-sphere. This gives the following corollary to Theorem \ref{main21032020}.

\begin{cor}
Let $\Omega\subset\R^n$ be an Euclidean domain with boundary $\Gamma$, then for every $k\geqslant 1$, one has

$$
\lambda_{W,k}^{\beta}(\Om)
\leqslant
\left[ A_n{i(\Gamma)}^{\frac{2}{n-1}}\beta+1\right]\left( \frac{k}{\vol(\Gamma)}  \right)^{\frac{2}{n-1}}+B(\Omega,\beta),
 $$
 where  the constant $A_n$ depends only on the dimension $n$ and  $B(\Omega,\beta)$ depends on the geometry ($n$, $\vol(\Omega)$, $\vol(\Gamma)$) and on $\beta$.\\
If in addition, $\Gamma$ is convex, we have $i(\Gamma) = 2$ and
$$
\lambda_{W,k}^{\beta}(\Om)
\leqslant
\left( A'_n \beta+1\right)\left( \frac{k}{\vol(\Gamma)}  \right)^{\frac{2}{n-1}}+B(\Omega,\beta),
 $$
 where $A'_n$ is a dimensional constant.
\end{cor}
%
%
%
\paragraph*{Plan of the paper}  In the next section, Section \ref{section07042020a}, we present and prove technical results which will be used in the sequel. Section \ref{section07042020b}  is devoted to proving our main results.
\section{Metric measure space decomposition}\label{section07042020a}
Much of what we do in this section carries over to the metric space setting and is inspired by \cite{ColboisMaerten}, \cite{hyp} and \cite{Asma2011}.
We adopt the notation $(X,d,\mu)$ to designate a metric measure space such that 
$X$ is  complete and  locally compact with respect to the distance $d$ and $\mu$ is a Borel measure supported in a bounded Borelian subset $\overline{Y}\subset X$, such that $\mu(\overline{Y}\backslash Y)=0$ and $\mu(Y)\subset (0,\infty)$. For every $x\in X$ and $r>0$, $B(x,r):=\{y\in X: d(x,y)<r\}$ designates the metric open ball.
We denominate capacitor any  couple $(A, B)$ of subsets such  that $\emptyset\neq A \subset B \subset X$. Two capacitors $(A_1, B_1)$ and $(A_2, B_2)$ are disjoint if $B_1\cap B_2=\emptyset$.  A family of capacitors is a finite set of capacitors in $X$ that are pairwise disjoint. 
\begin{defn}[Covering property]
Let $(X, d)$ be a complete, locally compact metric space. We denote by $\diam(X)$ the diameter, defined as the maximal distance between any two points of $X$. Let $\varepsilon>1$, $\rho>0$ and an  increasing function $N:(0,\rho]\longrightarrow\N_{\geqslant 2}$. We say that 
$(X, d)$ satisfies the $(N,\varepsilon;\rho)$-covering property if each ball of radius $r$ such that $0<r \leqslant\rho$ can be covered by $N(r)$ balls of radius $\frac{r}{\varepsilon}$. We shall omit the symbol $\varepsilon$ in the notation if and only if $\varepsilon=2$. In the same way, we drop the term  $\rho$ from the notation if and only if $\rho\geqslant\diam(X)$, which we call a global covering property.
\end{defn}

\begin{defn}(Measure radial monotonicity)
Let $(X, d,\mu)$ be  as described above.
We say that the measure $\mu$ is radially monotone in $X$ if
\begin{equation}
\underset{r\rightarrow 0}{\lim}\{ \sup \mu(B(x, r)),~ x \in X\} = 0. 
\end{equation}
\end{defn}
If $(X, d, \mu)$ is as above  and radially monotone and satisfies the covering property, then $\underset{r\rightarrow 0}{\lim} N^2\mu(B(x, r)) = 0$ for all $x \in X$. 
We require the following Lemma due to Colbois and Maerten in the proof of our main theorems.
\begin{lem}[Colbois-Maerten, $2008$]\label{CM}
Let  $(X, d, \mu)$ be as above  and radially monotone. Assume that  $(X, d, \mu)$ satisfies the $(N,4;1)$-covering property, with $N$ constant. Let $r>0$ and $K\in\N$ such that  for every $x\in X$, $\mu(B(x,r))\leqslant\frac{\mu(X)}{4KN_r^2} =:\alpha$. Then there exists a family of $K$ capacitors $\{(A_i,B_i)\}_{1\leqslant i\leqslant K}$ with the following properties for $1\leqslant i,j\leqslant K$:
\begin{enumerate}
\item $\mu(A_i)\geqslant2N\alpha$,
 \item $B_i=A_i^r:= \{x\in X,~d(x,A_i) < r\}$ is the $r$-neighbourhood of $A_i$ and  $d(B_i,B_j)>2r$ whenever $i\neq j$.
\end{enumerate} 
\end{lem}

\begin{defn}
Let $(X, d, \mu)$ be as above. The capacitors obtained when applying the construction of  Colbois-Maerten in Lemma \ref{CM} will be called CM-capacitors. Consistently, we call spherical capacitor   any capacitor  $(A, B)$  such  that $A$ and $B$ are both metric balls in $X$.
\end{defn}
 In the following lemma, we provide a useful procedure for construction  of a general family of (spherical or CM) capacitors. 
 \begin{lem}\label{lem25}
Let $(X, d,\mu)$ be  as we described above and $N\in \N$ such that $(X, d,\mu)$  satisfies the $(N,4; 1)$-covering property. Let $0<r_0\leqslant \frac{1}{10}$ be fixed, then for every $K\in \N$, $X$ satisfies at least one of the following properties: 
\begin{enumerate}
\item $X$ contains  a family of $K$ spherical capacitors
$\{(A_j,B_j)\}_{j=1}^K$  such that 
\begin{itemize}
\item $A_j=B(x_j,r_j)$ and $\mu(A_j)\geqslant \alpha $, with $ x_j\in X$, $r_j\in(0,2r_0]$,
 \item $B_j=B(x_j,2r_j)$.
\end{itemize}
\item $X$ contains  a family of $K$ CM-capacitors $\{(A_j,A_j^{\tilde{r}_0})\}_{j=1}^K$ where $\tilde{r}_0=\min\{r_0,\tau_1\}$ and 
\begin{equation}
\tau_1:=\sup\{r\in \R_{> 0} : \mu(B(x,r))\leqslant \alpha, \forall x\in X \}.
\end{equation}
\end{enumerate}
 \end{lem} 
 \begin{proof}
 The proof consists of two steps. The first part is  an iterative scheme presenting a method to construct the proof objects. And, if the assumptions to achieve the construction of the $K$ spherical capacitors in the first step do not hold, then we shall be able to use Lemma \ref{CM}, and solve the problem once in Step \ref{step2}.
\begin{proofpart} \label{step1}
  To start, define  
  $$X_1:=X,\quad \mu_1(A):=\mu(A),$$ for every measurable set $A\subset X$ and
  $$ \tau_1=\sup\{r\in \R_{> 0} : \mu_1(B(x,r))\leqslant \alpha, \forall~ x\in X_1\}.$$
We have the two possible cases:
\begin{enumerate}[label={}]
 \item \underline{Case  $\tau_1\geqslant r_0$:} Then $\mu_1(B(x,r_0))\leqslant  \alpha$, for every  $x\in X_1$. In this case we set $\tilde{r}_0:=r_0$ and move to Step \ref{step2}.    
    \item \underline{Case  $\tau_1< r_0$:} Set $r_1:=\frac{3}{2}\tau_1$. Then, there exists $x_1\in X_1$ such that  $\mu_1(B(x_1,r_1))>  \alpha$. We define 
    \begin{equation*}
A_1:=B(x_1,r_1),\quad B_1:=B(x_1,2r_1), \quad C_1:=B(x_1,4r_1). 
    \end{equation*} 
    We have then  $\mu(A_1)>  \alpha$.
    There are two important observations for the inductive step. First, $C_1$ can be covered by  $N^2$ balls of radius $\frac{r_1}{4}$ ($4r_1=6\tau_1<6r_0=6r_0<1$ then $C_1$ can be covered by  $N$ balls of radius $r_1$ and each of those balls can be covered by $N$ balls of radius $\frac{r_1}{4}$ ). Second, $\mu_1(B(x,\frac{r_1}{4}))\leqslant \alpha$ for all $x\in  X$, since $\frac{r_1}{4}<\tau_1$. Hence $\mu_1(C_1)\leqslant  N^2 \alpha$ and 
        \begin{equation*}
 \mu(X\backslash C_1)\geqslant \mu(X)-\mu(C_1)\geqslant \mu(X)\left(1-\frac{1}{4K}\right)> \frac{\mu(X)}{2}.   
        \end{equation*}    
\end{enumerate}  
 \paragraph*{First iteration.}
    We define  
  $$X_2:=X\backslash C_1,\quad \mu_2(A):=\mu(A\cap X_2),$$
  for every measurable set $A\subset X$ and
  $$ \tau_2:=\sup\{r\in \R_{> 0} : \mu_2(B(x,r))\leqslant \alpha, \forall~ x\in X_2\}.$$
   We have:
    \begin{enumerate}[label={}]
   \item \underline{Case  $\tau_2\geqslant \tau_1$:} Then $\mu_2(B(x, \tau_1))\leqslant\alpha$, for every $x\in X_2$. In this case we set $\tilde{r}_0:=\tau_1$ and move to Step \ref{step2}.    
    \item \underline{Case  $\tau_2<  \tau_1$:} Set $r_2:=\frac{3}{2}\tau_2$. Then, there exists $x_2\in X_2$ such that  $\mu_2(B(x_2,r_2))>  \alpha$. We define 
    $$A_2:=B(x_2,r_2),\quad B_2:=B(x_2,2r_2),$$ 
    $$C_2:=C_1\cup B(x_2,4r_1).$$  
    We have $\mu(A_2)=\mu(A_2\cap X)\geqslant\mu(A_2\cap X_2)=\mu_2(A_2) >  \alpha$.\\
    In addition, $r_2< r_1$ hence $B_1\cap B_2=\emptyset$.
    
Similarly, $B(x_2,4r_1)$ can be covered by  $N^2$ balls of radius $\frac{r_2}{4}$ ($4r_1=6\tau_1<6r_0<6r_0<1$ then $B(x_2,4r_1)$ can be covered by  $N$ balls of radius $r_2$ and each of those balls can be covered by $N$ balls of radius $\frac{r_2}{4}$ ). 

Since $\mu(B(x,\frac{r_2}{4}))\leqslant\mu(B(x,\frac{r_1}{4}))\leqslant \alpha$ for all $x\in  X$, $\mu(B(x_2,4r_1))\leqslant  N^2 \alpha$.
     Notice that 
   $\mu(C_2)\leqslant \mu(C_1)+\mu(B(x_2,4r_1))\leqslant 2N^2 \alpha$. Hence, one has
   \begin{equation*}
 \mu(X\backslash C_2)= \mu(X)-\mu(C_2)\geqslant \mu(X)\left(1-\frac{1}{2K}\right)> \frac{\mu(X)}{2}.   
        \end{equation*}        
   
    \end{enumerate}
 \paragraph*{Iteration $j$, with  $1 < j \leqslant K$.}
Suppose that we have constructed $j-1$ capacitors $\{(A_1,B_1),\ldots, (A_{j-1},B_{j-1})\}$ satisfying for $1\leqslant i\neq l\leqslant j-1$
 \begin{equation*}
 \begin{cases}
 \mu(A_i)>\alpha,\\
 B_i\cap B_l=\emptyset,\\
  C_{j-1}=\bigcup_{i=1}^{j-1}B(x_i,4r_1) \quad\text{ and }\quad \mu(C_{j-1})\leqslant (j-1) N^2 \alpha
 \end{cases}
 \end{equation*}
 with for every $1\leqslant i\leqslant j-1$
 \begin{equation*}
  \begin{cases}
  X_i= X\backslash C_{i-1},~ C_0:=\emptyset\\
  \mu_i(\cdot)=\mu(\cdot\cap X_i)\\
  \tau_i=\sup\{r\in \R_{> 0} : \mu_i(B(x,r))\leqslant \alpha, \forall~ x\in X_i\}\\
  0< \tau_i< \tau_1<r_0.
 \end{cases} 
 \end{equation*}

    Define  $X_j:=X\backslash C_{j-1}$, then
    \begin{align*}
\mu(X_j)&=\mu(X\backslash C_{j-1})=\mu(X)-\mu(C_{j-1})\\
&\geqslant \mu(X)\left(1-\frac{j-1}{4K}\right)>\frac{\mu(X)}{2}>0.
    \end{align*}
 Then, define the measure  $\mu_j(A):=\mu(A\cap X_j),$
  for all measurable set $A\subset X$ and 
  $$ \tau_j:=\sup\{r\in \R_{> 0} : \mu_j(B(x,r))\leqslant \alpha, \forall~ x\in X_j\}.$$

    \begin{enumerate}[label={}]    
    \item \underline{Case  $\tau_j\geqslant \tau_1$:} We have $\mu_j(B(x,\tau_1))\leqslant  \alpha$, for every $x\in X_j$. In this case we move to Step \ref{step2}.
    
    \item \underline{Case  $\tau_j< \tau_1$:} Set  $r_j:=\frac{3}{2}\tau_j$. Then, there exists $x_j\in X_j$ such that  $\mu_j(B(x_j,r_j))>  \alpha$. We define 
    $$A_j:=\mu_j(B(x_j,r_j)),\quad B_j:=\mu_j(B(x_j,2r_j)),$$
    $$ C_j:=C_{j-1}\cup B(x_j,4r_1)$$ 
  One has  $\mu(A_j)\geqslant\mu(A_j\cap X_2)=\mu_j(A_j) >  \alpha$.\\
Again $B(x_j,4r_1)$ can be covered by  $N^2$ balls of radius $\frac{r_1}{4}$,  hence 
$$\mu(B(x_j,4r_1))\leqslant  N^2 \alpha.$$
We have
     \begin{align*}
  \mu(C_j)&\leqslant\mu(C_{j-1})+\mu(B(x_j,4r_1))\\& \leqslant  jN^2 \alpha\leqslant  KN^2 \alpha.
     \end{align*} 
    Since for $i=1,\ldots,j$ we have $r_i<r_1$,  $B_j\cap B_i=\emptyset$ for all $1\leqslant i\leqslant j-1$ and
   $$\mu(X\backslash C_j)\geqslant \mu(X)\left(1- \frac{1}{4}\right)> \frac{\mu(X)}{2}. $$

%
    \end{enumerate}    
    \end{proofpart}
\begin{proofpart}\label{step2} 
Let  $\tilde{r}_0:=\min\{r_0,\tau_1\}$. We suppose that, for some $j_0\in \{1,\ldots, K-1\}$, we have $\mu_{j_0}(B(x,\tilde{r}_0))\leqslant  \alpha$, for every $ x\in X_{j_0}$. Clearly, if $j_0=1$ (respectively $j_0>1$), then $\tilde{r}_0=r_0$ (respectively $\tilde{r}_0=\tau_1$). In each case, $\mu_{j_0}(X_{j_0})=\mu(X_{j_0}>\frac{\mu(X)}{2}>0$. Applying Lemma \ref{CM} to $\left( X_{j_0},d,\mu_{j_0}\right)$, we obtain in  $X_{j_0}$ (then in $X$) a family of $K$ CM-capacitors $\{(A_j,A_j^{\tilde{r}_0})\}_{j=1}^K$. This concludes the proof.
      \qedhere
\end{proofpart}
 \end{proof}

\section{Proof of mains Theorems}\label{section07042020b}

In this section we prove our main results. We start with the following proposition which will be useful to prove Theorem \ref{main21032020}.
\begin{prop}\label{prop020432020}
Let $n\geqslant 2$ and $\tilde{C}$ be a positive real number. Let $(M, g)$ be a complete $n$-dimensional Riemannian manifold with $\ric (M,g) \geqslant -(n-1)\kappa^2$, with $\kappa\in \R_{\geqslant 0}$. Let $\Omega\subset M$ be a domain with smooth boundary $\Gamma$ such that $\Omega\in \mathcal{M}(n,\tilde{C})$. Then for every $k\geqslant 1$, one has

\begin{multline}
\lambda_{W,k}^{\beta}(\Om)
\leqslant
 A(n,\beta,\kappa,\Omega)\left( \frac{\tilde{C}k}{\vol(\Gamma)}  \right)^{\frac{2}{n-1}}
 \\
+B(n,\kappa,\Omega)
 \left( \frac{k}{\vol(\Gamma)}  \right)^{\frac{2}{n}} 
+C(n,\beta,\kappa,\Omega),
\end{multline}
where 
\begin{align*}
A(n,\beta,\kappa,\Omega)&=A(n,\kappa) \left[\frac{\vol(\Omega)}{\vol(\Gamma)}+\beta\right]\\
B(n,\kappa,\Omega)&=B(n,\kappa)
 \left( \frac{\vol(\Om)}{\vol(\Gamma)}  \right)^{1-\frac{2}{n}}\\
 C(n,\beta,\kappa,\Omega)&=C(n,\kappa) \left[  \frac{\vol(\Om)}{\vol(\Gamma)}+\beta \right],
\end{align*}
$A(n,\kappa)$, $B(n,\kappa)$ and $C(n,\kappa)$ depend only on the dimension $n$ and $\kappa$. 
\end{prop}

\begin{proof}
Consider the  metric measure space $(M,d, \mu)$, where  $d$ is the distance from the metric $g$ and $\mu$ is the Borel measure with support $\Gamma$ defined for each Borelian $A$ of  $M$ by
$$\mu(A):=\int_{A\cap\Gamma} \mathrm{dv}_g.$$
This measures the area of the part of the hypersurface $\Gamma$ lying inside the subset $A$. 

We will start by showing that the metric space $(M,d, \mu)$	 satisfies the assumptions of Lemma \ref{lem25}. Then, according to the nature of the capacitors obtained after applying Lemma \ref{lem25}, we define a disjointly supported family of test functions  and  bound their Rayleigh quotient. This allows us to conclude the proof using the variational characterisation of $\lambda_{W,k}^{\beta}(\Om)$.

Being a topological manifold, $M$ is locally compact and the radial monotonicity is fulfilled. Thanks to the Hopf-Rinow theorem  $(M,d)$ is a complete metric space. The measure $\mu$ is supported in $\Gamma$, we have clearly $\mu(\overline{\Gamma}\backslash \Gamma)=\mu(\emptyset)=0$ and $\mu(\Gamma)=\vol(\Gamma)\in (0,\infty)$.

To show that the metric space $(M,d)$ satisfies the $(N,4; 1)$-covering property for some constant $N\in\N_{\geqslant 2}$, we take $x\in M$, $0<r<1$ and $\{B(x_i,\frac{r}{8})\}_{i=1}^N$ a maximal family of disjoint balls with center $x_i\in B(x,r)$. By the maximality assumption, the family $\{B(x_i,\frac{r}{4})\}_{i=1}^N$ covers $B(x,r)$. To prove that $N$ is finite, take $i_0\in\{1,\ldots,N\}$ such that
 $$\vol \left(B\big(x_{i_0},\frac{r}{8}\big)\right)=\min_{1\leqslant i\leqslant N}\vol \left(B\big(x_i,\frac{r}{8}\big)\right).$$
  Then one has that $N\vol(B(x_{i_0},\frac{r}{8}))\leqslant \sum_{1\leqslant i\leqslant N}\vol(B(x_{i},\frac{r}{8}))$ since the balls  $B(x_{i},\frac{r}{8})$ are pairwise disjoint. In addition, $B(x_{i},\frac{r}{8})\subset B(x_{i},r+\frac{r}{8})$  for every $x_i\in B(x,r)$. Hence
$ N\vol(B(x_{i_0},\frac{r}{8}))\leqslant\vol(B(x,\frac{9r}{8}))$,
$$N\leqslant \frac{\vol(B(x,\frac{9r}{8}))}{\vol(B(x_{i_0},\frac{r}{8}))}
\leqslant \frac{\vol(B(x,2r))}{\vol(B(x_{i_0},\frac{r}{8}))}
\leqslant \frac{\vol(B(x_{i_0},4r))}{\vol(B(x_{i_0},\frac{r}{8}))}.$$
Using the volume comparison Theorem (Bishop 1964, Gromov 1980), one has 
$$\frac{\vol(B(x_{i_0},4r))}{\vol(B(x_{i_0},\frac{r}{8}))}\leqslant \frac{\nu(n,-\kappa^2,4r)}{\nu(n,-\kappa^2,\frac{r}{8})},$$
where $\nu(n,\kappa^2,r)$ denotes the volume of a ball of radius $r$ in the constant curvature model space $M_{\kappa^2}^n$. Then
\begin{align*}
N&\leqslant\frac{\int_0^{4r}\sinh^{n-1}(\kappa t)\mathrm{d}t}{\int_0^{\frac{r}{8}}\sinh^{n-1}(\kappa t)\mathrm{d}t}
\leqslant \frac{\int_0^{4r} \left[(\kappa t)e^{\kappa t}\right]^{n-1}\mathrm{d}t}{\int_0^{\frac{r}{8}}(\kappa t)^{n-1}\mathrm{d}t}\\
&\leqslant \frac{e^{4r(n-1)\kappa } \int_0^{4r} t^{n-1}\mathrm{d}t}{\int_0^{\frac{r}{8}}t^{n-1}\mathrm{d}t}=2^{5n}e^{4r(n-1)\kappa }<2^{5n}e^{4(n-1)\kappa }.
\end{align*}
Now set $K:=4k$ and $r_0:=\frac{1}{10}$. Applying Lemma \ref{lem25}, there exists in $M$ either a family $\mathfrak{B}=\{(A_j,B_j)\}_{j=1}^{4k}$ of spherical capacitors such that
\begin{itemize}
\item $A_j=B(x_j,r_j)$,~ $x_j\in X$,~ $r_j\in(0,2r_0]$,~ $\mu(A_j)\geqslant \alpha=\frac{\vol(\Gamma)}{16kN^2}$,
 \item $B_j=B(x_j,2r_j)$,
\end{itemize}
or a family $\mathfrak{C}=\{(A_j,A_j^{r_0})\}_{j=1}^{4k}$ of $4k$ CM-capacitors such that  $\mu(A_j)\geqslant 2N\alpha$.

\paragraph*{First case $M\supset \mathfrak{B}$.} For each $1\leqslant j \leqslant 4k$, we consider the function $f_j$ supported
in $B_j=B(x_j,2r_j)\in \mathfrak{B}$ and defined by 
$$f_j(x):=
\begin{cases}
\min \{1,2-\frac{d(x_j,x)}{r_j}\}\quad &\forall x\in B_j,\\
0\quad &\forall x\in M\backslash B_j.
\end{cases}
$$
One sees that
$$ R_\beta(f_j)\leqslant\frac{\int_{\Om\cap B_j}|\nabla f_j|^2 \mathrm{d}_M+\beta\int_{\Gamma\cap B_j}{|\nabla {f_j}|^2 \mathrm{d}_\Gamma}}{\int_{\Gamma\cap A_j}{{f_i}^2 \mathrm{d}_\Gamma}}.$$
\begin{enumerate}[label=\roman*)]
\item \label{ai} Since  for every $x\in A_j$, $f_j(x)=1$, one has
$$\int_{\Gamma\cap A_j}{{f_j}^2 \mathrm{d}_\Gamma}\geqslant\int_{\Gamma\cap A_j}\mathrm{d}_\Gamma\geqslant \mu(A_j)\geqslant \frac{\vol(\Gamma)}{16 N^2 k}. $$
\item\label{aii} Set for $x\in M$, $d_j(x):=d(x_j, x)$, then
$$|\nabla f_j|\leqslant \left|\nabla(2-\frac{d_j(x)}{r_j})\right|=\left|\frac{1}{r_j}\nabla(d_j(x))\right|\leqslant\frac{1}{r_j}.$$ 
 By H\"older's inequality, we have 
 \begin{align*}
 \int_{\Om\cap B_j}|\nabla f_j|^2 \mathrm{d}_M &\leqslant \left(\int_{\Om\cap B_j}|\nabla f_j|^n \mathrm{d}_M \right)^{\frac{2}{n}}\left(\int_{\Om\cap B_j}\mathrm{d}_M \right)^{1-\frac{2}{n}}\\
 &\leqslant \left(\frac{1}{r_j^n}\int_{\Om\cap B_j}1 \mathrm{d}_M \right)^{\frac{2}{n}}\left(\int_{\Om\cap B_j}\mathrm{d}_M \right)^{1-\frac{2}{n}}\\
 &\leqslant \left(\frac{1}{r_j^n}\vol(B_j) \right)^{\frac{2}{n}}\left( \vol(\Om\cap B_j) \right)^{1-\frac{2}{n}}.\\
 \end{align*}
 However, one has 
 \begin{align*}
  \vol (B_j)&\leqslant \nu(n,-\kappa^2, r_j)\leqslant \frac{2^n}{n}r_j^n e^{2(n-1)r_j\kappa }\\
  &\leqslant \frac{2^n}{n}r_j^n e^{2(n-1)\kappa }=:c(n,\kappa) r_j^n.
 \end{align*}
In addition, the $B_j$'s are pairwise disjoint then 
$$\sum_{j=1}^{4k}\vol(\Om\cap B_j)\leqslant \vol(\Om).$$
 We deduce that at least $2k$ of $B_j$'s satisfy
\begin{equation}
\vol(\Om\cap B_j)\leqslant\frac{\vol(\Om)}{k}.\label{alk1}
\end{equation} 
Up to re-ordering, we assume that for the first $2k$ of the $B_j$'s we have \eqref{alk1}.
Hence, 
$$ \int_{\Om\cap B_j}|\nabla f_j|^2 \leqslant c(n,\kappa)^{\frac{2}{n}}\left(\frac{\vol(\Om)}{k} \right)^{1-\frac{2}{n}},\quad \forall~ 1\leqslant j\leqslant 2k.$$
\item \label{aiii}
Notice that  $d(x_j,\Gamma)\leqslant 2r_j\leqslant 4 r_0<R_0$. We have
 \begin{equation}\label{assumpt25032020}
 \vol(\Gamma\cap B(x_j,2r_j))\leqslant \tilde{C} (2r_j)^{n-1}.
  \end{equation}
 and then
 \begin{align*}
 \int_{\Gamma\cap B_j}|\nabla f_j|^2 \mathrm{d}_M &\leqslant \left(\int_{\Gamma\cap B_j}|\nabla f_j|^{n} \mathrm{d}_M \right)^{\frac{2}{n-1}}\left(\int_{\Gamma\cap B_j}\mathrm{d}_M \right)^{1-\frac{2}{n-}}\\
 &\leqslant \left(\frac{1}{r_j^{n-1}}\int_{\Gamma\cap B_j} \mathrm{d}_M \right)^{\frac{2}{n-1}}\left(\int_{\Gamma\cap B_j}\mathrm{d}_M \right)^{1-\frac{2}{n-1}}\\
 &\leqslant \left(\frac{1}{r_j^{n-1}}\vol(\Gamma\cap B(x_j,2r_j)) \right)^{\frac{2}{n-1}}\left( \vol(\Gamma\cap B_j) \right)^{1-\frac{2}{n-1}}\\
 &\leqslant \left(2^{n-1} \tilde{C} \right)^{\frac{2}{n-1}}\left( \vol(\Gamma\cap B_j) \right)^{1-\frac{2}{n-1}}.
 \end{align*}
In addition, again the $B_j$'s are pairwise disjoint then 
$$\sum_{j=1}^{4k}\vol(\Gamma\cap B_j)\leqslant \vol(\Gamma).$$
 Hence at least $k$ of $B_j$'s satisfy
\begin{equation}\label{alk2}
\vol(\Gamma\cap B_j)\leqslant\frac{\vol(\Gamma)}{k}.
\end{equation} 
Up to re-ordering, we assume that for the first $k$ of the $B_j$'s inequality \eqref{alk2} holds.
Hence, 
 \begin{align*}
 \int_{\Gamma\cap B_j}|\nabla f_j|^2 \mathrm{d}_M 
 \leqslant 4\tilde{C}^{\frac{2}{n-1}}\left( \frac{\vol(\Gamma)}{k} \right)^{1-\frac{2}{n-1}}.
 \end{align*}
\end{enumerate}
Combining \ref{ai}, \ref{aii}, \ref{aiii}, one has
\begin{align}
 R_\beta(f_j)&\leqslant  \frac{16 N^2 k}{\vol(\Gamma)}
\Big[
 c(n,\kappa)^{\frac{2}{n}}\left(\frac{\vol(\Om)}{k} \right)^{1-\frac{2}{n}}
+4\beta \tilde{C}^{\frac{2}{n-1}}\left( \frac{\vol(\Gamma)}{k} \right)^{1-\frac{2}{n-1}}
 \Big] \nonumber\\
&\leqslant
B(n,\kappa)\left( \frac{k}{\vol(\Gamma)}  \right)^{\frac{2}{n}} 
 \left( \frac{\vol(\Om)}{\vol(\Gamma)}  \right)^{1-\frac{2}{n}}+A(n,\kappa)\beta\left( \frac{\tilde{C}k}{\vol(\Gamma)}  \right)^{\frac{2}{n-1}},
\label{estim1}
\end{align}
where
$  A(n,\kappa):=2^{16n}e^{8(n-1)\kappa }$ and
  $B(n,\kappa):=2^{14n}e^{8(n-1)\kappa } c(n,\kappa)^\frac{2}{n}$.
\paragraph*{Second case  $M\supset \mathfrak{C}$.}
For each $1\leqslant j \leqslant 4k$, we consider the function $\varphi_j$ supported
on $A_j^{\tilde{r}_0}$ defined by 
$$\varphi_j(x):=
\begin{cases}
1-\frac{d(A_j,x)}{\tilde{r}_0} \quad &\forall x\in A_j^{\tilde{r}_0},\\
0\quad &\forall x\in M\backslash A_j^{\tilde{r}_0}.
\end{cases}
$$
We have
$$ R_\beta(\varphi_j)\leqslant\frac{\int_{\Om\cap A_j^{\tilde{r}_0}}|\nabla \varphi_j|^2 \mathrm{d}_M+\beta\int_{\Gamma\cap A_j^{\tilde{r}_0}}{|\nabla {\varphi_j}|^2 \mathrm{d}_\Gamma}}{\int_{\Gamma\cap A_j}{{\varphi_i}^2 \mathrm{d}_\Gamma}}.$$
\begin{enumerate}[label=\roman*)]
\item \label{bi} Since  for every $x\in A_j$, $\varphi_j(x)=1$, one has
$$\int_{\Gamma\cap A_j}{{\varphi_j}^2 \mathrm{d}_\Gamma}\geqslant\int_{\Gamma\cap A_j}\mathrm{d}_\Gamma\geqslant \mu(A_j)\geqslant \frac{\vol(\Gamma)}{8N k} $$
\item\label{bii} We have
$$
 \int_{\Om\cap A_j^{\tilde{r}_0}}|\nabla \varphi_j|^2 \mathrm{d}_M \leqslant \frac{1}{\tilde{r}_0^2}\vol(\Om\cap A_j^{\tilde{r}_0}).
$$
The $A_j^{\tilde{r}_0}$'s are pairwise disjoint then $\sum_{j=1}^{4k}\vol(\Om\cap A_j^{\tilde{r}_0})\leqslant \vol(\Om)$.\\
 We deduce that at least $2k$ of $A_j^{\tilde{r}_0}$'s satisfies 
\begin{equation}
\vol(\Om\cap A_j^{\tilde{r}_0})\leqslant\frac{\vol(\Om)}{k}.\label{blk1}
\end{equation} 
Up to re-ordering, we assume that for the first $2k$ of the $A_j^{\tilde{r}_0}$'s we have \eqref{blk1}.
Hence, 
$$ \int_{\Om\cap A_j^{\tilde{r}_0}}|\nabla \varphi_j|^2 \leqslant\frac{1}{\tilde{r}_0^2} \frac{\vol(\Om)}{k},\quad \forall~ 1\leqslant j\leqslant 2k.$$
\item \label{biii}
By the same argument, at least $k$ of $A_j^{\tilde{r}_0}$'s satisfy
\begin{equation}
\vol(\Gamma\cap A_j^{\tilde{r}_0})\leqslant\frac{\vol(\Gamma)}{k}.\label{blk2}
\end{equation} 
Up to re-ordering, we assume that for the first $k$ of the $A_j^{\tilde{r}_0}$'s inequality \eqref{blk2} holds.
Hence, 
$$ \int_{\Gamma\cap A_j^{\tilde{r}_0}}|\nabla \varphi_j|^2\leqslant \frac{1}{\tilde{r}_0^2} \frac{\vol(\Gamma)}{k},\quad \forall~ 1\leqslant j\leqslant k.$$
\end{enumerate}
Combining \ref{bi}, \ref{bii}, \ref{biii}, one has
\begin{align*}
 R_\beta(\varphi_j)&\leqslant  \frac{2N k}{\vol(\Gamma)}\left[ \frac{1}{\tilde{r}_0^2} \frac{\vol(\Om)}{k}+\beta \frac{1}{\tilde{r}_0^2} \frac{\vol(\Gamma)}{k}   \right]\\
&= \frac{2N}{\tilde{r}_0^2} \left[  \frac{\vol(\Om)}{\vol(\Gamma)}+\beta \right].
\end{align*}

Hence, if $\tilde{r}_0=r_0$, then
\begin{equation}\label{estim2}
 R_\beta(\varphi_j)\leqslant C(n,\kappa) \left[  \frac{\vol(\Om)}{\vol(\Gamma)}+\beta \right],
\end{equation}
where $C(n,\kappa):=\frac{2^{5n+3}e^{4(n-1)\kappa } }{r_0^2}= 5^2\cdot 2^{5(n+1})e^{4(n-1)\kappa } $. Otherwise, $\tilde{r}_0=\tau_1$ and there exists $x\in X$ such that $\mu\left(B(x,2\tilde{r}_0)\right)>\alpha$. Take $y\in B(x,2\tilde{r}_0)\cap \Gamma$, since $B(y,4\tilde{r}_0)\supset B(x,2\tilde{r}_0)$, we have $\tilde{C}(4\tilde{r}_0)^{n-1}\geqslant \mu\left(B(y,4\tilde{r}_0)\right)> \alpha$. 
Consequently, $\frac{1}{\tilde{r}_0^2}\leqslant 4\left(\frac{16\tilde{C} kN^2}{\vol(\Gamma)} \right)^{\frac{2}{n-1}}$ and we have
\begin{equation}\label{estim2prime}
 R_\beta(\varphi_j)\leqslant A'(n,\kappa)\left[ \frac{\vol(\Om)}{\vol(\Gamma)}+\beta \right]\left(\frac{\tilde{C}k}{\vol(\Gamma)} \right)^{\frac{2}{n-1}},
\end{equation}
with $A'(n,\kappa)$ depending on $n$ and $\kappa$.

In both cases $ R_\beta(\varphi_j)$ is bounded from above by the sum of the right-hand
sides in \eqref{estim1}, \eqref{estim2} and \eqref{estim2prime}. Without loss of generality, one can assume that $A(n,\kappa)\geqslant A'(n,\kappa)$. One concludes the argument by applying the min-max characterization of $\lambda_{W,k}^{\beta}(\Om)$.
\end{proof}

\begin{proof}[Proof of Theorem \ref{main21032020}]
From Proposition \ref{prop020432020}, for
every $k\geqslant 1$, we have
\begin{align}
\lambda_{W,k}^{\beta}(\Om)
&\leqslant
A(n,\kappa)\left[\frac{\vol(\Omega)}{\vol(\Gamma)}+\beta\right]\left( \frac{\tilde{C}k}{\vol(\Gamma)}  \right)^{\frac{2}{n-1}}
\nonumber \\
&+B(n,\kappa)
 \left( \frac{\vol(\Om)}{\vol(\Gamma)}  \right)^{1-\frac{2}{n}}
 \left( \frac{k}{\vol(\Gamma)}  \right)^{\frac{2}{n}} 
\nonumber\\
&+C(n,\kappa) \left[  \frac{\vol(\Om)}{\vol(\Gamma)}+\beta \right],
\end{align}

where the constants $A(n,\kappa)$, $B(n,\kappa)$ and $C(n,\kappa)$ are in the form $c(n)e^{\kappa }$, $c(n)$ being a term involving $n$ and $\kappa$ free.

 \begin{enumerate}[label={}]
    
    \item[-] If $\kappa\leqslant 1$, then $A(n,\kappa)$, $B(n,\kappa)$ and $B'(n,\kappa)$ can be replace by constants depending only on $n$:
    \begin{align*}
\lambda_{W,k}^{\beta}(\Om)
&\leqslant
A(n)\left[\frac{\vol(\Omega)}{\vol(\Gamma)}+\beta\right]\left( \frac{\tilde{C}k}{\vol(\Gamma)}  \right)^{\frac{2}{n-1}}
 \\
&+B(n)
 \left( \frac{\vol(\Om)}{\vol(\Gamma)}  \right)^{1-\frac{2}{n}}
 \left( \frac{k}{\vol(\Gamma)}  \right)^{\frac{2}{n}} 
\\
&+C(n) \left(  \frac{\vol(\Om)}{\vol(\Gamma)}+\beta \right).
\end{align*}
       
    \item[-] Otherwise, we assume that  $\ric (M,g) \geqslant -(n-1)\kappa^2 g$ with $\kappa> 1$. Then the Ricci curvature $\ric (M,\tilde{g})$ of the rescaled metric $\tilde{g}:=\kappa^2 g$ is bounded by $-(n-1)\tilde{g}$. We mark with a tilde quantities associated with the metric ${\tilde {g}}$, while those unmarked with such will be still associated with the metric $g$. Then we have
    
        \begin{align*}
\tilde{\lambda}_{W,k}^{\beta}(\Om)
&\leqslant
A(n)\left[\frac{\vol(\Omega)}{\vol(\Gamma)}+\beta\right]\left( \frac{\tilde{C}k}{\mathrm{Vol}_{\tilde{g}}(\Gamma)}  \right)^{\frac{2}{n-1}}
 \\
&+B(n)
 \left( \frac{\mathrm{Vol}_{\tilde{g}}(\Om)}{\mathrm{Vol}_{\tilde{g}}(\Gamma)}  \right)^{1-\frac{2}{n}}
 \left( \frac{k}{\mathrm{Vol}_{\tilde{g}}(\Gamma)}  \right)^{\frac{2}{n}} 
\\
&+C(n) \left(  \frac{\mathrm{Vol}_{\tilde{g}}(\Om)}{\mathrm{Vol}_{\tilde{g}}(\Gamma)}+\beta \right),
\end{align*}
In addition, since $\kappa>1$, for all  $u\in \mathfrak{V}_\beta\backslash\{0\}$ we have \begin{equation*}
\tilde{R}_\beta(u)=\frac{\kappa\int_\Om{|\nabla u|^2 \mathrm{d}_M+\beta\int_{\Gamma}{|\nabla_\Gamma u|^2 \mathrm{d}_\Gamma}}}{\kappa^2\int_{\Gamma}{u^2 \mathrm{d}_\Gamma}}\geqslant\frac{1}{\kappa^2} {R}_\beta(u).
\end{equation*}
Every orthonormal basis of  a $ k$-dimensional  subspaces $V\in\mathfrak{V}(k) $ of   $ \mathfrak{V}_\beta$ remains orthogonal with the metric $\tilde{g}$, then using the variation characterisation, we have
\begin{align*}
\lambda_{W,k}^{\beta}(\Om)\leqslant\kappa^2\tilde{\lambda}_{W,k}^{\beta}(\Om)
&\leqslant
A(n)\left[\frac{{\mathrm{Vol}_{\tilde{g}}(\Omega)}}{{\mathrm{Vol}_{\tilde{g}}(\Gamma)}}+\beta\right]\kappa^2\left( \frac{\tilde{C}k}{\mathrm{Vol}_{\tilde{g}}(\Gamma)}  \right)^{\frac{2}{n-1}}
 \\
&+B(n)\kappa^2
 \left( \frac{\mathrm{Vol}_{\tilde{g}}(\Om)}{\mathrm{Vol}_{\tilde{g}}(\Gamma)}  \right)^{1-\frac{2}{n}}
 \left( \frac{k}{\mathrm{Vol}_{\tilde{g}}(\Gamma)}  \right)^{\frac{2}{n}} 
\\
&+C(n) \kappa^2\left(  \frac{\mathrm{Vol}_{\tilde{g}}(\Om)}{\mathrm{Vol}_{\tilde{g}}(\Gamma)}+\beta \right).
\end{align*}
However $\mathrm{Vol}_{\tilde{g}}(\Om)=\kappa^{n}\vol(\Om)$ and $\mathrm{Vol}_{\tilde{g}}(\Gamma)=\kappa^{n-1}\vol(\Gamma)$, thus
    \begin{align*}
\lambda_{W,k}^{\beta}(\Om)
&\leqslant
A(n)\left[\kappa\frac{\vol(\Omega)}{\vol(\Gamma)}+\beta\right]\left( \frac{\tilde{C}k}{\vol(\Gamma)}  \right)^{\frac{2}{n-1}}
 \nonumber\\
&+B(n)\kappa
 \left( \frac{\vol(\Om)}{\vol(\Gamma)}  \right)^{1-\frac{2}{n}}
 \left( \frac{k}{\vol(\Gamma)}  \right)^{\frac{2}{n}} 
\nonumber\\
&+C(n) \kappa^2\left(  \kappa\frac{\vol(\Om)}{\vol(\Gamma)}+\beta \right)\nonumber\\
&\leqslant
A(n)\left[\kappa\frac{\vol(\Omega)}{\vol(\Gamma)}+\beta\right]\left( \frac{\tilde{C}k}{\vol(\Gamma)}  \right)^{\frac{2}{n-1}}
 \nonumber\\
&+B(n)\kappa
 \left( \frac{\vol(\Om)}{\vol(\Gamma)}  \right)^{1-\frac{2}{n}}
 \left( \frac{k}{\vol(\Gamma)}  \right)^{\frac{2}{n}} 
\nonumber\\
&+C(n) \kappa^3\left(  \frac{\vol(\Om)}{\vol(\Gamma)}+\beta \right).
\end{align*}
 \end{enumerate}
 In each case,
  \begin{align}
\lambda_{W,k}^{\beta}(\Om)
&\leqslant
A(n)\left[\kappa\frac{\vol(\Omega)}{\vol(\Gamma)}+\beta\right]\left( \frac{\tilde{C}k}{\vol(\Gamma)}  \right)^{\frac{2}{n-1}}
 \nonumber\\
&+B(n)\big(\kappa+1\big)
 \left( \frac{\vol(\Om)}{\vol(\Gamma)}  \right)^{1-\frac{2}{n}}
 \left( \frac{k}{\vol(\Gamma)}  \right)^{\frac{2}{n}} 
\nonumber\\
&+C(n) \big(\kappa^3+1\big)\left(  \frac{\vol(\Om)}{\vol(\Gamma)}+\beta \right).\label{eq07042020}
\end{align}
The result follows setting $B(n,\kappa):=B(n)\big(\kappa+1\big)$ and $C(n,\kappa):=C(n)\big(\kappa^3+1\big)$.
 \end{proof}

 \begin{proof}[Proof of Corollary \ref{cor070420202}]
We rewrite the second term in the right hand side of \eqref{eq07042020} that we refer as $T_2$:
 $$T_2=\frac{\overline{B}}{k^{\frac{2}{n(n-1)}}} \left( \frac{k}{\vol(\Gamma)}  \right)^{\frac{2}{n-1}}$$
 where
$ \overline{B}:=
 B(n,\kappa)
 \left( \frac{\vol(\Om)}{\vol(\Gamma)}  \right)^{1-\frac{2}{n}}
\vol(\Gamma)^{\frac{2}{n(n-1)}}$.
\begin{enumerate}
\item If $k\leqslant \overline{B}^{\frac{n(n-1)}{2}}$ then $T_2$ is bounded from above by
$$\frac{\overline{B}^{1+\frac{2}{n}}}{\vol(\Gamma)^{\frac{2}{n-1}}}
= B(n,\kappa)^{1+\frac{2}{n}}\frac{\vol(\Om)^{1-\frac{4}{n^2}}}{\vol(\Gamma)^{1-\frac{2}{n(n-1)}+\frac{2}{n+1}-\frac{4}{n^2(n-1)}-\frac{4}{n^2}}}
=:\overline{\overline{B}},$$
which is a geometric constant free from $k$. 

\item
Otherwise, we have $\frac{\overline{B}}{k^{\frac{2}{n(n-1)}}}<1$ and then
$$T_2<\left( \frac{k}{\vol(\Gamma)}  \right)^{\frac{2}{n-1}}.$$
\end{enumerate}
In each case, we have
$$T_2\leqslant \left( \frac{k}{\vol(\Gamma)}  \right)^{\frac{2}{n-1}}+ \overline{\overline{B}}.$$
Hence replacing in \eqref{eq07042020} and setting $$B(\Omega,\beta):=\overline{\overline{B}}+C(n,\kappa)\left(  \frac{\vol(\Om)}{\vol(\Gamma)}+\beta \right),$$ we get 
$$
\lambda_{W,k}^{\beta}(\Om)
\leqslant
\left( A(n)\left[\kappa\frac{\vol(\Omega)}{\vol(\Gamma)}+\beta\right]\tilde{C}^{\frac{2}{n-1}}+1\right)\left( \frac{k}{\vol(\Gamma)}  \right)^{\frac{2}{n-1}}+B(\Omega,\beta).
 $$
\end{proof}

The following lemma, from \cite[Lemma 3.2]{Nardulli2018} gives a volume estimate result which will be very useful to achieve our estimate in Theorem \ref{main29}.

 \begin{lem}[Nardulli, 2018]\label{Lnard}
 Let $(M,g)$ be a Riemannian manifold of dimension $n\geqslant 2$ and $\Gamma\subset M$ a smooth hypersurface.
Then there exist two constants $R_0 > 0$ and $C_0 > 0$ such that for every $x\in M$  at distance $d$ from $\Gamma$, one has 
\begin{equation}
\vol_g(\Gamma\cap B(x,R))\leqslant(1+2RC_0)\omega_{n-1}(2R)^{n-1},\quad\forall~R\in[d, R_0),
\end{equation}
where $\omega_{n-1}$ is the volume of the unit ball of $\R^{n-1}$.
Here $R_0$ is a constant depending on geometric data of $\Gamma$ and the ambient Riemannian manifold $M$ including a bound on the second fundamental form, normal injectivity radius of $\Gamma$ and injectivity radius of $M$. As well, the constant $C_0$ depends on the same quantities but also on a lower bound on the Ricci curvature of $\Gamma$.
 \end{lem}
 \begin{rem}
 In the statement of \cite[Lemma 3.2]{Nardulli2018} the inequality holds for every $x\in M$ with $\dist(\Gamma,x)<R_0$ and $R<R_0$. One can consider $R \in[\dist(\Gamma,x), R)$ since in the case $\dist(\Gamma,x)>R$, the intersection $\Gamma\cap B(x,r)$ is empty and the inequality is trivial. Assuming that $\dist(\Gamma,x)\leqslant R < R_0$, we have in the right hand side  $\dist(\Gamma,x)+ R\leqslant 2R$ which leads to our statement.
 
The proof in \cite{Nardulli2018} reduces the problem to an application of  Bishop-Gromov inequality  estimating the volume of an intrinsic ball of $\Gamma$. This is done by using comparison theorems for distortion of the normal exponential map based on a submanifold, to compare the extrinsic and intrinsic  distance functions on $\Gamma$.
\end{rem}

\begin{proof}[Proof of Theorem \ref{main29}]
 The proof follows along the same lines as the proofs of Proposition \ref{prop020432020} with only slight modifications.
We consider the  metric measure space $(M,d, \mu)$, where  $d$ is the distance from the metric $g$ and $\mu$ is the Borel measure with support $\Gamma$ defined for each Borelian $A$ of  $M$ by
$$\mu(A):=\int_{A\cap\Gamma} \mathrm{dv}_g.$$

It satisfies the assumptions of Lemma \ref{lem25}  as we have already seen in the proof of  Proposition \ref{prop020432020}.  
.

We set $K:=4k$ and $r_0:=\frac{1}{10}\min\{1,R_0\}$ where $R_0$ is the same constant as in Lemma \ref{Lnard}.  Applying Lemma \ref{lem25}, there exists in $M$ either a family $\mathfrak{B}=\{(A_j,B_j)\}_{j=1}^{4k}$ of spherical capacitors such that
\begin{itemize}
\item $A_j=B(x_j,r_j)$,~ $x_j\in X$,~ $r_j\in(0,2r_0]$,~ $\mu(A_j)\geqslant \alpha=\frac{\vol(\Gamma)}{16kN^2}$,
 \item $B_j=B(x_j,2r_j)$,
\end{itemize}
or a family $\mathfrak{C}=\{(A_j,A_j^{r_0})\}_{j=1}^{4k}$ of $4k$ CM-capacitors such that  $\mu(A_j)\geqslant 2N\alpha$.
\paragraph*{First case $M\supset \mathfrak{B}$.} 
 This first part of the proof is exactly the same as the proof of Theorem \ref{main21032020} until the point \eqref{aiii25032020} below.
For each $1\leqslant j \leqslant 4k$, we consider the function $f_j$ supported
in $B_j=B(x_j,2r_j)\in \mathfrak{B}$ and defined by 
$$f_j(x):=
\begin{cases}
\min \{1,2-\frac{d(x_j,x)}{r_j}\}\quad &\forall x\in B_j,\\
0\quad &\forall x\in M\backslash B_j.
\end{cases}
$$
One sees that
$$ R_\beta(f_j)\leqslant\frac{\int_{\Om\cap B_j}|\nabla f_j|^2 \mathrm{d}_M+\beta\int_{\Gamma\cap B_j}{|\nabla {f_j}|^2 \mathrm{d}_\Gamma}}{\int_{\Gamma\cap A_j}{{f_i}^2 \mathrm{d}_\Gamma}}.$$
\begin{enumerate}[label=\roman*)]
\item \label{ai25032020} Since  for every $x\in A_j$, $f_j(x)=1$, one has
$$\int_{\Gamma\cap A_j}{{f_j}^2 \mathrm{d}_\Gamma}\geqslant\int_{\Gamma\cap A_j}\mathrm{d}_\Gamma\geqslant \mu(A_j)\geqslant \frac{\vol(\Gamma)}{16 N^2 k}. $$
\item\label{aii25032020} Set for $x\in M$, $d_j(x):=d(x_j, x)$, then
$$|\nabla f_j|\leqslant \left|\nabla(2-\frac{d_j(x)}{r_j})\right|=\left|\frac{1}{r_j}\nabla(d_j(x))\right|\leqslant\frac{1}{r_j}.$$ 
 By H\"older's inequality, we have 
 \begin{align*}
 \int_{\Om\cap B_j}|\nabla f_j|^2 \mathrm{d}_M &\leqslant \left(\int_{\Om\cap B_j}|\nabla f_j|^n \mathrm{d}_M \right)^{\frac{2}{n}}\left(\int_{\Om\cap B_j}\mathrm{d}_M \right)^{1-\frac{2}{n}}\\
 &\leqslant \left(\frac{1}{r_j^n}\int_{\Om\cap B_j}1 \mathrm{d}_M \right)^{\frac{2}{n}}\left(\int_{\Om\cap B_j}\mathrm{d}_M \right)^{1-\frac{2}{n}}\\
 &\leqslant \left(\frac{1}{r_j^n}\vol(B_j) \right)^{\frac{2}{n}}\left( \vol(\Om\cap B_j) \right)^{1-\frac{2}{n}}.
 \end{align*}
 However, one has 
 \begin{align*}
  \vol (B_j)&\leqslant \nu(n,-\kappa^2, r_j)\leqslant \frac{2^n}{n}r_j^n e^{2(n-1)r_j\kappa }\\
  &\leqslant \frac{2^n}{n}r_j^n e^{2(n-1)\kappa }=:c(n,\kappa) r_j^n.
 \end{align*}
In addition, the $B_j$'s are pairwise disjoint then 
$$\sum_{j=1}^{4k}\vol(\Om\cap B_j)\leqslant \vol(\Om).$$
 We deduce that at least $2k$ of $B_j$'s satisfy
\begin{equation}
\vol(\Om\cap B_j)\leqslant\frac{\vol(\Om)}{k}.\label{alk250320201}
\end{equation} 
Up to re-ordering, we assume that the first $2k$ of the $B_j$'s satisfy \eqref{alk250320201}. Hence, 
$$ \int_{\Om\cap B_j}|\nabla f_j|^2 \leqslant c(n,\kappa)^{\frac{2}{n}}\left(\frac{\vol(\Om)}{k} \right)^{1-\frac{2}{n}},\quad \forall~ 1\leqslant j\leqslant 2k.$$
\item \label{aiii25032020}
Notice that  $d(x_j,\Gamma)\leqslant 2r_j\leqslant 4 r_0<R_0$. Applying Lemma \ref{Lnard} with $R=2r_j$, we have
\begin{equation*}
 \vol(\Gamma\cap B(x_j,2r_j))\leqslant(1+4r_jC_0)\omega_{n-1}(4r_j)^{n-1}.
\end{equation*}
Either $1\geqslant 4r_jC_0$ and then one has
 \begin{align*}
 \int_{\Gamma\cap B_j}|\nabla f_j|^2 \mathrm{d}_M &\leqslant \left(\int_{\Gamma\cap B_j}|\nabla f_j|^{n-1} \mathrm{d}_M \right)^{\frac{2}{n-1}}\left(\int_{\Gamma\cap B_j}\mathrm{d}_M \right)^{1-\frac{2}{n-1}}\\
 &\leqslant \left(\frac{1}{r_j^{n-1}}\int_{\Gamma\cap B_j} \mathrm{d}_M \right)^{\frac{2}{n-1}}\left(\int_{\Gamma\cap B_j}\mathrm{d}_M \right)^{1-\frac{2}{n-1}}\\
 &\leqslant \left(\frac{1}{r_j^{n-1}}\vol(\Gamma\cap B(x_j,2r_j)) \right)^{\frac{2}{n-1}}\left( \vol(\Gamma\cap B_j) \right)^{1-\frac{2}{n-1}}\\
 &\leqslant \left(2^{2n-1}\omega_{n-1} \right)^{\frac{2}{n-1}}\left( \vol(\Gamma\cap B_j) \right)^{1-\frac{2}{n-1}}.
 \end{align*}
Or, $1\leqslant 4r_jC_0$ and then
 \begin{align*}
 \int_{\Gamma\cap B_j}|\nabla f_j|^2 \mathrm{d}_M &\leqslant \left(\int_{\Gamma\cap B_j}|\nabla f_j|^{n} \mathrm{d}_M \right)^{\frac{2}{n}}\left(\int_{\Gamma\cap B_j}\mathrm{d}_M \right)^{1-\frac{2}{n}}\\
 &\leqslant \left(\frac{1}{r_j^{n}}\int_{\Gamma\cap B_j}1 \mathrm{d}_M \right)^{\frac{2}{n}}\left(\int_{\Gamma\cap B_j}\mathrm{d}_M \right)^{1-\frac{2}{n}}\\
 &\leqslant \left(\frac{1}{r_j^{n}}\vol(\Gamma\cap B(x_j,2r_j)) \right)^{\frac{2}{n}}\left( \vol(\Gamma\cap B_j) \right)^{1-\frac{2}{n}}\\
 &\leqslant \left(2^{3+2(n-1)}C_0\omega_{n-1} \right)^{\frac{2}{n}}\left( \vol(\Gamma\cap B_j) \right)^{1-\frac{2}{n}}.
 \end{align*}
In each case, 
 \begin{align*}
 \int_{\Gamma\cap B_j}|\nabla f_j|^2 \mathrm{d}_M 
 \leqslant & \left(2^{2n-1}\omega_{n-1} \right)^{\frac{2}{n-1}}\left( \vol(\Gamma\cap B_j) \right)^{1-\frac{2}{n-1}}\\
 &+ \left(2^{3+2(n-1)}C_0\omega_{n-1} \right)^{\frac{2}{n}}\left( \vol(\Gamma\cap B_j) \right)^{1-\frac{2}{n}}.
 \end{align*}
In addition, again the $B_j$'s are pairwise disjoint then 
$$\sum_{j=1}^{4k}\vol(\Gamma\cap B_j)\leqslant \vol(\Gamma).$$
 Hence at least $k$ of $B_j$'s satisfy
\begin{equation}\label{alk250320202}
\vol(\Gamma\cap B_j)\leqslant\frac{\vol(\Gamma)}{k}.
\end{equation} 
Up to re-ordering, we assume that for the first $k$ of the $B_j$'s inequality \eqref{alk250320202} holds.
Hence, 
 \begin{align*}
 \int_{\Gamma\cap B_j}|\nabla f_j|^2 \mathrm{d}_M 
 \leqslant & \left(2^{2n-1}\omega_{n-1} \right)^{\frac{2}{n-1}}\left( \frac{\vol(\Gamma)}{k}  \right)^{1-\frac{2}{n-1}}\\
 &+ \left(2^{3+2(n-1)}C_0\omega_{n-1} \right)^{\frac{2}{n}}\left( \frac{\vol(\Gamma)}{k} \right)^{1-\frac{2}{n}}.
 \end{align*}
\end{enumerate}
Combining \ref{ai25032020}, \ref{aii25032020}, \ref{aiii25032020}, one has

\begin{align*}
 R_\beta(f_j)\leqslant & \frac{16 N^2 k}{\vol(\Gamma)}
\Big[
 c(n,\kappa)^{\frac{2}{n}}\left(\frac{\vol(\Om)}{k} \right)^{1-\frac{2}{n}}\\
&+\beta\left(2^{2n-1}\omega_{n-1} \right)^{\frac{2}{n-1}}\left( \frac{\vol(\Gamma)}{k}  \right)^{1-\frac{2}{n-1}}\\
 &+\beta\left(2^{3+2(n-1)}C_0\omega_{n-1} \right)^{\frac{2}{n}}\left( \frac{\vol(\Gamma)}{k} \right)^{1-\frac{2}{n}}
 \Big].
 \end{align*}
 Thus,
 \begin{align}
 R_\beta(f_j)\leqslant &  A(n,\kappa)\beta  \left( \frac{k}{\vol(\Gamma)}  \right)^{\frac{2}{n-1}} \nonumber \\
& + B(n,\kappa,C_0)\left[\frac{\vol(\Omega)}{\vol(\Gamma)} +\beta\right]   \left( \frac{k}{\vol(\Gamma)}  \right)^{\frac{2}{n}},\label{estim250320201}
\end{align}
where
\begin{align*}
A(n,\kappa)&:=2^{28n}\omega_{n-1}^{\frac{2}{n-1}}e^{8(n-1)\kappa }\\
(n,\kappa,C_0)&:=2^{24n}(C_0\omega_{n-1})^{\frac{2}{n}}e^{8(n-1)\kappa }+2^{24n}e^{12(n-1)\kappa }.
\end{align*}

\paragraph*{Second case  $M\supset \mathfrak{C}$.}
For each $1\leqslant j \leqslant 4k$, we consider the function $\varphi_j$ supported
on $A_j^{\tilde{r}_0}$ defined by 
\begin{equation*}
\varphi_j(x):=
\begin{cases}
1-\frac{d(A_j,x)}{r_j} \quad &\forall x\in A_j^{\tilde{r}_0},\\
0\quad &\forall x\in M\backslash A_j^{\tilde{r}_0}.
\end{cases}
\end{equation*}
We have
$$ R_\beta(\varphi_j)\leqslant\frac{\int_{\Om\cap A_j^{\tilde{r}_0}}|\nabla \varphi_j|^2 \mathrm{d}_M+\beta\int_{\Gamma\cap A_j^{\tilde{r}_0}}{|\nabla {\varphi_j}|^2 \mathrm{d}_\Gamma}}{\int_{\Gamma\cap A_j}{{\varphi_i}^2 \mathrm{d}_\Gamma}}.$$
\begin{enumerate}[label=\roman*)]
\item \label{bi25032020} Since  for every $x\in A_j$, $\varphi_j(x)=1$, one has
$$\int_{\Gamma\cap A_j}{{\varphi_j}^2 \mathrm{d}_\Gamma}\geqslant\int_{\Gamma\cap A_j}\mathrm{d}_\Gamma\geqslant \mu(A_j)\geqslant \frac{\vol(\Gamma)}{8N k} $$
\item\label{bii25032020} We have
$$
 \int_{\Om\cap A_j^{\tilde{r}_0}}|\nabla \varphi_j|^2 \mathrm{d}_M \leqslant \frac{1}{\tilde{r}_0^2}\vol(\Om\cap A_j^{\tilde{r}_0}).
$$
The $A_j^{\tilde{r}_0}$'s are pairwise disjoint then $\sum_{j=1}^{4k}\vol(\Om\cap A_j^{\tilde{r}_0})\leqslant \vol(\Om)$.\\
 We deduce that at least $2k$ of $A_j^{\tilde{r}_0}$'s satisfies 
\begin{equation}
\vol(\Om\cap A_j^{\tilde{r}_0})\leqslant\frac{\vol(\Om)}{k}.\label{blk250320201}
\end{equation} 
Up to re-ordering, we assume that for the first $2k$ of the $A_j^{\tilde{r}_0}$'s we have \eqref{blk250320201}.
Hence, 
$$ \int_{\Om\cap A_j^{\tilde{r}_0}}|\nabla \varphi_j|^2 \leqslant\frac{1}{\tilde{r}_0^2} \frac{\vol(\Om)}{k},\quad \forall~ 1\leqslant j\leqslant 2k.$$
\item \label{biii25032020}
By the same argument, at least $k$ of $A_j^{\tilde{r}_0}$'s satisfy
\begin{equation}
\vol(\Gamma\cap A_j^{\tilde{r}_0})\leqslant\frac{\vol(\Gamma)}{k}.\label{blk250320202}
\end{equation} 
Up to re-ordering, we assume that for the first $k$ of the $A_j^{\tilde{r}_0}$'s inequality \eqref{blk250320202} holds.
Hence, 
$$ \int_{\Gamma\cap A_j^{\tilde{r}_0}}|\nabla \varphi_j|^2\leqslant \frac{1}{\tilde{r}_0^2} \frac{\vol(\Gamma)}{k},\quad \forall~ 1\leqslant j\leqslant k.$$
\end{enumerate}
Combining \ref{bi25032020}, \ref{bii25032020}, \ref{biii25032020}, one has
\begin{align}
 R_\beta(\varphi_j)&\leqslant  \frac{8N k}{\vol(\Gamma)}\left[ \frac{1}{\tilde{r}_0^2} \frac{\vol(\Om)}{k}+\beta \frac{1}{\tilde{r}_0^2} \frac{\vol(\Gamma)}{k}   \right]\nonumber\\
&= \frac{2N}{\tilde{r}_0^2} \left[  \frac{\vol(\Om)}{\vol(\Gamma)}+\beta \right].\label{eq26042020}
\end{align}
Hence, if $\tilde{r}_0=r_0$, then
\begin{equation}
 R_\beta(\varphi_j)= C \left[  \frac{\vol(\Om)}{\vol(\Gamma)}+\beta \right],
\end{equation}
where $C=C(n,\kappa,R_0):=\frac{2^{10n}e^{4(n-1)\kappa } }{r_0^2} $ and $r_0:=\frac{1}{10}\min\{1,R_0\}$.\\
Otherwise, $\tilde{r}_0=\tau_1<r_0$ and there exist $x\in X$ such that $\mu(B(x,2\tau_1))>\alpha=\frac{\vol(\Gamma)}{16N^2k}$. Using Lemma \ref{Lnard}, we have
$$\left( 1+4\tilde{r}_0 C_0\right)\omega_{n-1}(4\tilde{r}_0)^{n-1}>\alpha.$$
Either  $4 \tau_1 C_0<1$  then  $\frac{1}{\tilde{r}_0}\leqslant a'(n,\kappa)\left(\frac{k}{\vol(\Gamma)} \right)^{\frac{2}{n-1}}$
where $a'(n,\kappa)$ depends only on $n$ and $\kappa$. Or, $4 \tau_1 C_0\geqslant 1$ and  $\frac{1}{\tilde{r}_0}\leqslant b'(n,\kappa,C_0)\left(\frac{k}{\vol(\Gamma)} \right)^{\frac{2}{n}}$
where $b'(n,\kappa,C_0)$ depends  on $n$, $\kappa$ and $C_0$. Replacing in\eqref{eq26042020}, these partial results are combined by a global upper bound after summation
\begin{multline}
 R_\beta(\varphi_j)\leqslant  A'(n,\kappa) \left[  \frac{\vol(\Om)}{\vol(\Gamma)}+\beta \right]\left(\frac{k}{\vol(\Gamma)} \right)^{\frac{2}{n-1}}\\
+B'(n,\kappa,C_0) \left[  \frac{\vol(\Om)}{\vol(\Gamma)}+\beta \right]\left(\frac{k}{\vol(\Gamma)} \right)^{\frac{2}{n}} \\
+ C(n,\kappa,R_0) \left[  \frac{\vol(\Om)}{\vol(\Gamma)}+\beta \right],\label{estim250320202}
\end{multline}
where $A'(n,\kappa)$, $B'(n,\kappa,C_0)$ and $C(n,\kappa,R_0)$ depend on the respective terms in parentheses.

We can assume that $A(n,\kappa)\geqslant A'(n,\kappa)$ and $B(n,\kappa,C_0)\geqslant B'(n,\kappa,C_0)$. In both cases $ R_\beta(\varphi_j)$ is bounded from above by the sum of the right-hand
sides in \eqref{estim250320201} and \eqref{estim250320202}. One concludes the argument by applying the min-max characterization of $\lambda_{W,k}^{\beta}(\Om)$:
\begin{align}
\lambda_{W,k}^{\beta}(\Om)
\leqslant  A(n,\kappa) \left[  \frac{\vol(\Om)}{\vol(\Gamma)}+\beta \right]\left(\frac{k}{\vol(\Gamma)} \right)^{\frac{2}{n-1}}\\
+B(n,\kappa,C_0) \left[ \left( \frac{\vol(\Om)}{\vol(\Gamma)}\right)^{1-\frac{2}{n}} +\frac{\vol(\Om)}{\vol(\Gamma)}+\beta \right]\left(\frac{k}{\vol(\Gamma)} \right)^{\frac{2}{n}} \\
+ C(n,\kappa,R_0) \left[  \frac{\vol(\Om)}{\vol(\Gamma)}+\beta \right],
\end{align}
where  the constants $B(n,\kappa,C_0)$ and $C(n,\kappa,R_0)$ depend on $n$ and geometric quantities $\kappa$, $C_0$ and $R_0$ respectively.

With the same arguments as in the prof of Theorem \ref{main21032020}, we have 
\begin{multline}
\lambda_{W,k}^{\beta}(\Om)
\leqslant
   A(n) \left[  \kappa\frac{\vol(\Om)}{\vol(\Gamma)}+\beta \right]\left(\frac{k}{\vol(\Gamma)} \right)^{\frac{2}{n-1}}\\
+\overline{B}(n,\kappa,C_0) \left[ \left( \frac{\vol(\Om)}{\vol(\Gamma)}\right)^{1-\frac{2}{n}} +\frac{\vol(\Om)}{\vol(\Gamma)}+\beta \right]\left(\frac{k}{\vol(\Gamma)} \right)^{\frac{2}{n}} \\
+ \overline{C}(n,\kappa,R_0) \left[  \frac{\vol(\Om)}{\vol(\Gamma)}+\beta \right],
\end{multline}
tant $A(n)$ depends only on the dimension $n$, the constants $\overline{B}(n,\kappa,C_0)$ and $\overline{C}(n,\kappa,R_0)$ depend on $n$ and geometric quantities $\kappa$, $C_0$ and $R_0$ respectively.
\end{proof}
\begin{proof}[Proof of Corollary \ref{cor07042020}]
The proof of Corollary \ref{cor07042020} is similar to the proof of Corollary \ref{cor070420202}
\end{proof}
 \bibliographystyle{plain}

\end{document}